\documentclass[10pt]{amsart}

\usepackage{amsmath,xspace,amssymb,mathrsfs}
\usepackage{color}

\input xy
\xyoption{all}
\xyoption{2cell}
\UseAllTwocells
\CompileMatrices

\newcommand{\Spec}{\operatorname{Spec}}
\renewcommand{\phi}{\varphi}

\newcommand{\Max}{\operatorname{Max}}

\newcommand{\Min}{\operatorname{Min}}

\newcommand{\Proj}{\operatorname{Proj}}

\newcommand{\Clop}{\operatorname{Clop}}
\newcommand{\Fin}{\operatorname{Fin}}

\newtheorem{proposition}{Proposition}[section]
\newtheorem{lemma}[proposition]{Lemma}

\newtheorem{corollary}[proposition]{Corollary}
\newtheorem{theorem}[proposition]{Theorem}

\theoremstyle{definition}
\newtheorem{definition}[proposition]{Definition}
\newtheorem{example}[proposition]{Example}

\newtheorem{remark}[proposition]{Remark}

\usepackage{etoolbox}
\makeatletter
\patchcmd{\@settitle}{\uppercasenonmath\@title}{}{}{}
\patchcmd{\@setauthors}{\MakeUppercase}{}{}{}
\makeatother

\begin{document}

\title[Connected components of schemes]{Connected components of qcqs schemes and projective spaces}
\author[A. Tarizadeh]{Abolfazl Tarizadeh}
\address{Department of Mathematics, Faculty of Basic Sciences, University of Maragheh, Maragheh, East Azerbaijan Province, Iran.}
\email{ebulfez1978@gmail.com}

\date{}

\begin{abstract} In this article, we first prove a general result in topology which states that every quasi-component of a quasi-spectral 
space is connected. \\
As an application, the structure of the connected components of every quasi-compact quasi-separated (qcqs) scheme $X$ is fully characterized. They are exactly of the form $f^{-1}(C)$ where $f:X\rightarrow\Spec(R)$ is the canonical morphism, $C$ is a connected component of $\Spec(R)$ and $R=\mathscr{O}_{X}(X)$ is the ring of global sections of $X$. \\ 
Next, we make new advances in understanding the structure of the connected components of projective spaces. In general, for an $\mathbb{N}$-graded ring $R=\bigoplus\limits_{n\geqslant0}R_{n}$, the structure of the connected components of scheme $\Proj(R)$ is still unknown. However, we show that for any scheme $S$ the connected components of the projective space $\mathbb{P}^{n}_{S}=
\mathbb{P}^{n}_{\mathbb{Z}}\times_{\Spec(\mathbb{Z})}S$ are exactly of the form $\mathbb{P}^{n}_{C}$ where $C$ is a connected component of $S$ which is equipped with a closed subscheme structure. \\ 
Finally, we characterize the finiteness of the number of connected components of a (quasi-compact) topological space in terms of purely algebraic conditions, in which the primitive idempotents play an important role in this characterization. 
\end{abstract}

\maketitle

\section{Introduction}

The subject of this article lies at the intersection of algebraic geometry, commutative algebra, and topology. 

A classical result in topology asserts that  every quasi-component of a compact space is connected (in other words, the connected components of a compact space are exactly its quasi-components). By compact space we mean a quasi-compact and Hausdorff topological space. In a topological space, by quasi-component we mean the intersection of all clopen (both open and closed) subsets containing a point.  In addition to the quasi-compactness, the Hausdorff assumption and hence normality play a major role in proving this fundamental theorem (see e.g. \cite[Exercise I.10.1]{Bredon} or \cite[Theorem 6.1.23]{Engelking}). 

On the other hand, an affine scheme $\Spec(R)$ is not necessarily Hausdorff or normal but its connected components still have the above property. Note that $\Spec(R)$ is Hausdorff if and only if $\dim(R)=0$. More generally, $\Spec(R)$ is a normal space if and only if $R$ is a Gelfand ring, i.e., for every prime ideal of $R$ there exists a unique maximal ideal of $R$ containing it (see e.g. \cite[Theorem 4.3]{Tarizadeh 3}). 

Inspired by the above observations, our first goal in this article is to investigate general quasi-compact spaces that are not necessarily Hausdorff but whose connected components still have the above property. We know that non-Hausdorff spaces naturally arise in many fields of mathematics, especially commutative algebra and algebraic geometry. In fact, we prove the following general theorem which is one of the main results of this article:

\begin{theorem} Every quasi-component of a quasi-spectral space is connected. 
\end{theorem}

This result, beside the above classical theorem (which asserts that every quasi-component of a compact space is connected), is the most general result in topology, especially in the context of connectedness of all quasi-components.

There is an important connection between the algebraic concept of ``idempotent" and the topological concept of ``clopen" at the level of affine schemes. Then we observed that this correspondence actually holds very broadly for any scheme. More precisely, we show that if $X$ is a scheme or more generally a locally ringed space and $R=\mathscr{O}_{X}(X)$ is the ring of its global sections then $\Clop(X)$, the Boolean ring of all clopen subsets of $X$, is canonically isomorphic to the Boolean ring of idempotents of $R$.

As a main application of the above results, the structure of the connected components of an important and extensive class of schemes is fully characterized:

\begin{theorem} Let $X$ be a qcqs scheme and $R=\mathscr{O}_{X}(X)$. Then the connected components of $X$ are exactly of the form $f^{-1}(C)$ where $f:X\rightarrow\Spec(R)$ is the canonical morphism and $C$ is a connected component of $\Spec(R)$. In particular, we have a natural homeomorphism $\pi_{0}(X)\simeq\pi_{0}(\Spec(R))$.  
\end{theorem}
 
Next in \S 4, we make new advances in understanding the structure of the connected components of other types of schemes, especially certain projective spaces. 

In general, for an arbitrary $\mathbb{N}$-graded ring $R=\bigoplus\limits_{n\geqslant0}R_{n}$, the structure of the connected components of scheme $\Proj(R)$ is still unknown. In fact, to our knowledge, there is not even a precise formulation for this problem (it is still unclear how the form of every connected component of $\Proj(R)$ looks like). 

However, for any scheme $S$, by using the connectivity of the fibers of the structure morphism $\mathbb{P}_{S}^{n}=
\mathbb{P}_{\mathbb{Z}}^{n}\times_{\Spec(\mathbb{Z})}
S\rightarrow S$ as well as its closedness property, we obtained the following result: 

\begin{theorem} For any scheme $S$, the connected components of the projective space $\mathbb{P}_{S}^{n}$ are exactly of the form $\mathbb{P}_{C}^{n}$
where $C$ is a connected component of $S$ which is equipped with a closed subscheme structure.
\end{theorem}

In particular, for any ring $R$, the connected components of $\mathbb{P}_{R}^{n}=\Proj R[x_{0},\ldots,x_{n}]$ are exactly of the form $\mathbb{P}_{R/I}^{n}=
\Proj(R/I)[x_{0},\ldots,x_{n}]$ where $\Spec(R/I)=V(I)$ is a connected component of $\Spec(R)$.

Then in \S5, we characterize the finiteness of the number of connected components of general topological spaces. Among the main results, we prove that a quasi-compact space $X$ has finitely many connected components if and only if there is an isomorphism of (Boolean) rings $\Clop(X)\simeq(\mathbb{Z}/2)^{\kappa}$ where $\kappa$ is the cardinal of the connected components of $X$ (note that $\kappa$ is not assumed to be finite here). 

Examining the connected components of spaces also leads us to an interesting result which, in particular, asserts that a commutative (nonzero) ring $R$ has finitely many idempotents if and only if there is an isomorphism of rings $R\simeq\prod\limits_{i=1}^{n}R/(1-e_{i})$, or equivalently,  $R\simeq\prod\limits_{i\in\kappa}R/(1-e_{i})$ with $e_i$ is a primitive idempotent of $R$ and $\kappa$ is
the cardinal (possibly infinite) of the connected components of $\Spec(R)$.

\section{Preliminaries}

We collect in this section some basic background for the reader’s convenience.

Let $X$ be a topological space. By $\pi_0(X)$ we mean the set of connected components of $X$ equipped with the quotient topology (which is a totally disconnected space).  
By $\Clop(X)$ we mean the set of all clopen (both open and closed) subsets of $X$. It can be seen that $\Clop(X)$ is a Boolean ring with addition $A+B=(A\cup B)\setminus(A\cap B)$ and multiplication $A\cdot B=A\cap B$. In fact, $\Clop(X)$ is a subring of the power set ring $\mathscr{P}(X)$. If $f:X\rightarrow Y$ is a continuous map of topological spaces, then the map $\Clop(f):\Clop(Y)\rightarrow\Clop(X)$ given by $V\mapsto f^{-1}(V)$ is a morphism of Boolean rings. These define a contravariant functor from the category of topological spaces to the category of Boolean rings 
(for more information see e.g. \cite{Tarizadeh}). 

The intersection of all clopen subsets of a topological space $X$ containing a point $x\in X$ is called the \emph{quasi-component} of that point and we denote it by $q(x)$. It is clear that if $y\in q(x)$, then $q(x)=q(y)$. In fact, the $q(x)$ with $x\in X$ are indeed the equivalence classes of the following equivalence relation (defined over $X$): We say that $x\sim y$ if $f(x)=f(y)$ for any continuous function $f:X\rightarrow\{0,1\}$ where $\{0,1\}$ is equipped with the discrete topology.

It is clear that a topological space $X$ is connected if and only if $X$ is nonempty and has no  nontrivial clopens, or equivalently, the Boolean ring $\Clop(X)$ is isomorphic to $\mathbb{Z}/2=\{0,1\}$, the ring of integers modulo 2. 

In a topological space, a connected subset is contained in a clopen if and only if they meet each other. In particular, every connected component $C=[x]$ of a topological space $X$  is contained in the quasi-component $q(x)$. In fact,  every quasi-component of a topological space $X$ is the disjoint union of the connected components of $X$ that meet it. 

By irreducible space we mean a nonempty topological space such that it cannot be written as the union of two proper closed subsets. Every maximal element of the set of irreducible subsets of a topological space $X$ is called an irreducible component of $X$. It is clear that every irreducible space is connected.  

In this article, all rings are assumed to be commutative. If $r$ is an element of a ring $R$, then $D(r)=\{\mathfrak{p}\in\Spec(R): r\notin\mathfrak{p}\}$ and $V(r)=\Spec(R)\setminus D(r)$. 

\section{Connected components of qcqs schemes}

For the first main result of this section we need the following notion:

\begin{definition} By \emph{quasi-spectral space} we mean a quasi-compact space $X$ with a basis consisting of  quasi-compact open subsets of $X$ such that the intersection of any two members of the basis is quasi-compact. 
\end{definition}

We know that the union of every finitely many quasi-compact subspaces of a topological space is quasi-compact. Hence, the intersection of any two quasi-compact open subsets of a quasi-spectral space is quasi-compact. 

By spectral space we mean a quasi-spectral space that is also sober, i.e., every irreducible closed subset has a unique generic point.
 
For any commutative ring $R$, the space $\Spec(R)$ is spectral. More generally, the underlying space of a scheme $X$ is spectral if and only if $X$ is a qcqs scheme (see the proof of Corollary \ref{Coro iv}).

\begin{example} Here we give examples of quasi-spectral spaces that are not spectral. Every infinite set $X$ equipped with the cofinite topology is a quasi-spectral space, but not spectral. Indeed, every open subset of $X$ is quasi-compact. This shows that $X$ is quasi-spectral. However, the whole space $X$ is irreducible closed with no generic point. 
Hence, $X$ is not spectral. \\
Every finite space is also quasi-spectral, but not necessarily spectral. For instance, $X=\{1,2,3\}$ with the topology $T=\{\emptyset,X, \{1,2\}\}$ is not a sober space, because $\overline{\{1\}}=X=\overline{\{2\}}$. \\
For the third example, first recall that by primary ideal we mean a proper ideal $I$ of a ring $R$ such that if $ab\in I$ with $a\notin\sqrt{I}$ for some $a,b\in R$, then $b\in I$. \\  
Let $X$ be the set of all primary ideals of a ring $R$. It is clear that $\Spec(R)\subseteq X$ and we have a map $X\rightarrow\Spec(R)$ given by $I\mapsto\sqrt{I}$. So there exists a (unique) topology over $X$ such that its basis opens are of the form $U_r=\{I\in X: r\notin\sqrt{I}\}$ with $r\in R$. It is clear that $U_1=X$ and $U_r\cap U_{r'}=U_{rr'}$ for all $r,r'\in R$. Every basis open $U_{r}$ is quasi-compact, because if $U_r=\bigcup\limits_{k\in S}U_{r_{k}}$ then $r\in\sqrt{(r_k: k\in S)}$ and so there exists some finite subset $S'\subseteq S$ with $r\in\sqrt{(r_k: k\in S')}$, this yields that $U_r=\bigcup\limits_{k\in S'}U_{r_{k}}$ and so $U_r$ is quasi-compact. Hence, $X$ is quasi-spectral. The closed subsets of $X$ are precisely of the form $F(J)=\{I\in X:J\subseteq\sqrt{I}\}$ where $J$ is an ideal of $R$. For any point $I\in X$ we have $\overline{\{I\}}=F(I)$. If $M$ is a maximal ideal of $R$ then $\overline{\{M\}}$ is the set of all ideals $I$ of $R$ with $\sqrt{I}=M$. In particular, $M^n\in \overline{\{M\}}$ for all $n\geqslant1$. If $M^{d}\neq M^{d+1}$ for some $d\geqslant1$, then they are distinct generic points of $\overline{\{M\}}$.
Hence, $X$ is not necessarily spectral. 
It can be also seen that every irreducible closed subset of $X$ has a generic point. The space $X$ in general does not necessarily have a closed point. If $I$ is a closed point of $X$ then $I$ is a maximal ideal of $R$. But we observed that the converse does not hold. If $R$ is a von-Neumann regular ring, then the closed points of $X$ are precisely the prime (maximal) ideals of $R$.     
\end{example}

Now we prove the first main result of this article: 

\begin{theorem}\label{Theorem I T-D-F} Every quasi-component of a quasi-spectral space is connected. 
\end{theorem}

\begin{proof} Let $X$ be a quasi-spectral space and let $F=\bigcap\limits_{\substack{A\in\Clop(X),
\\x\in A}}A$ be a quasi-component of $X$, i.e., $F$ is the intersection of all clopen subsets of $X$ containing a point $x\in X$. Suppose that $F$ is the union of disjoint closed subsets $F'$ and $F''$. We may assume $x\in F'$. It will be enough to show that $F'$ is the intersection of some clopen subsets of $X$ containing $x$ (if the case, then we will have $F\subseteq F'$ and so $F''$ will be empty). Clearly $F'$ is an open subset of $F$ and so $F'=F\cap U'$ for some open subset $U'$ of $X$. We may write $U'=\bigcup\limits_{i}U_i$ where each $U_i$ is a basis open and so is a quasi-compact open subset of $X$. 
But $F$ and so $F'$ are closed subsets of $X$, and hence $F'$ is quasi-compact in $X$. Thus $F'$ is contained in the union of finitely many of the $U_i$, i.e., $F'\subseteq \bigcup\limits_{i=1}^{n}U_i$. But $U=\bigcup\limits_{i=1}^{n}U_i$ is a quasi-compact open subset of $X$ and $F'=F\cap U$. Similarly, there exists a quasi-compact open subset $V$ in $X$ with $F''=F\cap V$. We have $F\cap U\cap V=F'\cap F''=\emptyset$ and so $U\cap V\subseteq X\setminus F=\bigcup\limits_{\substack{A\in\Clop(X),
\\x\in A}}(X\setminus A)$. By hypothesis, $U\cap V$ is quasi-compact, so there exist finitely many clopens $A_{1},\ldots, A_{m}$ in $X$ containing $x$ such that $U\cap V\subseteq\bigcup\limits_{i=1}^{m}(X\setminus A_i)$. This yields that $(\bigcap\limits_{i=1}^{m}A_i)\cap U\cap V=\emptyset$. We also have $F\subseteq U\cup V$. But $X\setminus(U\cup V)$ is a closed subset of $X$ and hence is quasi-compact. Thus there exist finitely many clopens $B_{1},\ldots, B_{n}$ in $X$ containing $x$ such that $\bigcap\limits_{k=1}^{n}B_k\subseteq U\cup V$. Then the clopen $W=(\bigcap\limits_{i=1}^{m}A_i)
\cap(\bigcap\limits_{k=1}^{n}B_k)$ is the union of disjoint open subsets $W\cap U$ and $W\cap V$. Hence, $W\cap U$ is a closed subset of $W$ and so it is a closed subset of $X$. Thus $W\cap U$ is a clopen subset of $X$ containing $x$. Then $F'=\bigcap\limits_{\substack{A\in\Clop(X),
\\x\in A}}(A\cap W\cap U)$ is the intersection of the clopen subsets $A\cap W\cap U$ of $X$ containing $x$. This completes the proof. 
\end{proof} 

The following result is one of the important applications of Theorem \ref{Theorem I T-D-F}.

\begin{corollary}\label{Coro iv} Every quasi-component of a qcqs scheme is connected. 
\end{corollary} 

\begin{proof} The underlying space of every scheme has a basis consisting of affine opens, and every affine scheme is quasi-compact. It is well known that a scheme $X$ is quasi-separated if and only if the intersection of any two quasi-compact (affine) open subsets of $X$ is quasi-compact. For its proof see \cite[Prop 1.2.7 ]{Grothendieck}. Hence, the underlying space of a qcqs scheme is spectral. Then the assertion follows from Theorem \ref{Theorem I T-D-F}. 
\end{proof}

One of the main results of \cite[Theorem 6]{Hochster} (see also \cite[\S 12.6.12]{Dickmann}) asserts that every spectral space is homeomorphic to the prime spectrum of a commutative ring, but this is quite difficult to prove. However, without needing this result, we obtain the following conclusion:

\begin{corollary}\label{Corollary I} Every quasi-component of a spectral space is connected. 
\end{corollary}

\begin{proof} It follows immediately from Theorem \ref{Theorem I T-D-F}. 
\end{proof}

\begin{remark}\label{Remark lrs 2} Recall that if $X$ is a locally ringed space then for any   $s\in R=\mathscr{O}_{X}(X)$ the set $X_{s}=\{x\in X: s_{x}\notin\mathfrak{m}_{x}\}$ is an open subset of $X$ and $s_{|X_{s}}$ is an invertible element of the ring $\mathscr{O}_{X}(X_{s})$ where $s_{x}$ is the germ of $s$ at the point $x$ and 
$\mathfrak{m}_{x}$ is the maximal ideal of the local ring $\mathscr{O}_{X,x}$. 
There is also a canonical morphism of locally ringed spaces $f:X\rightarrow\Spec(R)$ such that the map between the underlying spaces is given by $f(x)=\mathfrak{p}_{x}$ where $\mathfrak{p}_{x}$ is the contraction of $\mathfrak{m}_{x}$ under the canonical ring map $R=\mathscr{O}_{X}(X)\rightarrow\mathscr{O}_{X,x}$. The continuity of $f$ follows from the fact that $f^{-1}(D(s))=X_{s}$ for all $s\in R$. Then  $X_{s}\cap X_{t}=X_{st}$ for all $s,t\in R$. In particular, if $s$ is nilpotent then $X_{s}=\emptyset$. It can be seen that $s\in R$ is invertible if and only if $X_{s}=X$.
\end{remark}

Recall that for any (commutative) ring $R$, the set of its idempotents $\mathscr{B}(R)=\{e\in R: e=e^{2}\}$ with the addition $e\oplus e'=e+e'-2ee'$ and multiplication $e\cdot e'=ee'$ is a ring. We call $\mathscr{B}(R)$ the Boolean ring of $R$. Note that $\mathscr{B}(R)$ is a subset of $R$, but in general $\mathscr{B}(R)$ is not a subring of $R$.  

There is an important connection between the algebraic concept of idempotent and the topological concept of clopen, which is widely used in commutative algebra and algebraic geometry. This result asserts that the idempotents of a commutative ring $R$ are naturally in one-to-one correspondence with the clopen subsets of $\Spec(R)$. In fact, the map $\mathscr{B}(R)\rightarrow\Clop(\Spec(R))$ given by $e\mapsto D(e)$ is an isomorphism of Boolean rings (\cite[Theorem 3.1]{Tarizadeh}). The following lemma generalizes this result to any scheme, even any locally ringed space:   

\begin{lemma}\label{Theorem lrs 5} Let $X$ be a locally ringed space and $R=\mathscr{O}_{X}(X)$. Then the map $\mathscr{B}(R)\rightarrow\Clop(X)$ given by $s\mapsto X_{s}$ is an isomorphism of Boolean rings.  
\end{lemma}

\begin{proof} First we show that the above map is indeed a morphism of rings. For any ring $R$, it can be seen that the map $\mathscr{B}(R)\rightarrow\Clop(\Spec(R))$ given by $s\mapsto D(s)$ is a morphism of rings (see \cite[Theorem 3.1]{Tarizadeh}). The induced map $\Clop(f):\Clop(\Spec(R))\rightarrow\Clop(X)$ is also a ring map where $f:X\rightarrow\Spec(R)$ is the canonical morphism. Hence, their composition $\mathscr{B}(R)\rightarrow\Clop(X)$ given by $s\mapsto X_{s}$ is a ring map. \\
For injectivity, assume $U:=X_{s}=X_{t}$ for some idempotents $s,t\in R$. Then by Remark \ref{Remark lrs 2}, the idempotents $s_{|U}$ and $t_{|U}$ are invertible in $\mathscr{O}_{X}(U)$ and so $s_{|U}=1=t_{|U}$. If $x\in X\setminus U$ then the germs $s_{x}$ and $t_{x}$ are idempotent elements of the maximal ideal of the local ring $\mathscr{O}_{X,x}$ and so $s_{x}=0=t_{x}$. This shows that for each $x\in X\setminus U$ there exists an open neighbourhood $V_{x}\subseteq X$ of $x$ such that $s_{|V_{x}}=0=t_{|V_{x}}$. It is clear that $U$ together with $\{V_{x}: x\in X\setminus U\}$ is an open covering of $X$. This shows that $s=t$. \\
To prove surjectivity, take a clopen subset $V$ in $X$. Then the ring map $R=\mathscr{O}_{X}(X)\rightarrow\mathscr{O}_{X}(V)\times
\mathscr{O}_{X}(V^{c})$ given by $s\mapsto (s_{|V}, s_{|V^{c}})$ is an isomorphism of rings where $V^{c}=X\setminus V$. So there is an idempotent $t\in R$ such that $t_{|V}=1$  and $t_{|V^{c}}=0$. Then, using Remark \ref{Remark lrs 2}, it can be easily seen that $V=f^{-1}(D(t))=X_{t}$. This completes the proof.    
\end{proof}

\begin{corollary}\label{Coro lrs 9} Let $X$ be a locally ringed space and $R=\mathscr{O}_{X}(X)$. Then $X$ is connected if and only if $\Spec(R)$ is connected.
\end{corollary}

\begin{proof} This is an immediate consequence of Lemma \ref{Theorem lrs 5}. 
\end{proof}

In this article, by regular ideal of a ring $R$ we mean an ideal of $R$ that is generated by a set of idempotent elements of $R$. Every finitely generated regular ideal is a principal ideal (generated by an idempotent element). 

Every maximal element of the set of proper regular ideals of $R$ is called a max-regular ideal of $R$. 

If $P$ is a prime ideal of $R$ then the ideal $P^{\ast}=(e\in P: e=e^2)$ generated by all idempotent elements in $P$ is a max-regular ideal of $R$. 
This, in particular, shows that every proper regular ideal of $R$ is contained in a max-regular ideal of $R$. It can be also seen that every max-regular ideal of $R$ is of the form $P^{\ast}$ for some prime ideal $P$ of $R$.
We have then the following results:

\begin{lemma}\label{Lemma ii} Let $I$ be a proper regular ideal of a ring $R$. Then $I$ is a max-regular ideal of $R$ if and only if $R/I$ has no nontrivial idempotents. 
\end{lemma}

\begin{proof} Assume $I$ is a max-regular ideal of $R$. If $x+I$ is an idempotent in $R/I$ for some $x\in R$ then $x-x^2\in I$ and so $x-x^2=re$ for some idempotent $e\in I$ and some $r\in R$. It follows that $(x-x^2)(1-e)=re(1-e)=0$ and so $x(1-e)=x^2(1-e)$. This shows that $e':=x(1-e)$ is idempotent. But $I=P^{\ast}$ for some prime ideal $P$ of $R$. Thus $e'\in I$ or $1-e'\in I$. If $e'\in I$ then $x\in I$. If $1-e'\in I$ then $1-x\in I$. Hence, $R/I$ has no nontrivial idempotents. The reverse implication is clear. 
\end{proof}

\begin{corollary}\label{Lemma iii} Let $R$ be a ring. Then every quasi-component of $\Spec(R)$ is connected. 
\end{corollary}

\begin{proof} It follows immediately from Theorem \ref{Theorem I T-D-F}. We also provide alternative direct proof for this result. Let $F=\bigcap\limits_{\substack{A\in\Clop(X),
\\P\in A}}A$ be the intersection of all clopen subsets of $X=\Spec(R)$ containing a point $P\in\Spec(R)$. It is well known that every clopen subset of $\Spec(R)$ is of the form $V(e)=D(1-e)$ for some idempotent $e \in R$ (see e.g. Lemma \ref{Theorem lrs 5}). It follows that $F=V(P^{\ast})$. But $P^{\ast}$ is a max-regular ideal of $R$. Thus by Lemma \ref{Lemma ii}, $R/P^{\ast}$ has no nontrivial idempotents and so $F=V(P^{\ast})\simeq\Spec(R/P^{\ast})$ is connected. 
\end{proof}

The above result shows that the connected components of an affine scheme $\Spec(R)$ are precisely of the form $V(I)=\Spec(R/I)$ where $I$ is a max-regular ideal of $R$ (i.e., $I$ is a proper ideal of $R$ generated by a set of idempotent elements such that $R/I$ has no nontrivial idempotents). The more general case of this observation will be proved in Theorem \ref{Theorem lrs 44}.

Recall that if $I$ is an ideal of a ring $R$, then we always have $\Min(R)\cap V(I)\subseteq\Min(I)$. In this regard, we have the following result:

\begin{corollary} If $I$ is a max-regular ideal of a ring $R$, then $\Min(I)=\Min(R)\cap V(I)$.  
\end{corollary}

\begin{proof} If $\mathfrak{p}\in\Min(I)$ then there exists a minimal prime $\mathfrak{q}$ of $R$ with $\mathfrak{q}\subseteq\mathfrak{p}$. Since $V(\mathfrak{q})$ is connected (even irreducible), it is contained in a connected component $V(J)$ of $\Spec(R)$. But $\mathfrak{p}\in V(I)\cap V(J)$ and so $V(I)=V(J)$. This shows that $I\subseteq\mathfrak{q}$ and so $\mathfrak{p}=\mathfrak{q}\in\Min(R)$. 
\end{proof}

The converse of the above result is not true. For example, consider $I$ as the zero ideal.

\begin{remark}\label{Remark cc 3} Let $f:X\rightarrow Y$ be a continuous map of topological spaces such that for each connected component $C$ of $Y$ the set $f^{-1}(C)$ is a connected subset of $X$. Then the map $g:\pi_{0}(Y)\rightarrow\pi_{0}(X)$ given by $C\mapsto f^{-1}(C)$ is bijective. In fact, the inverse of $g$ is the natural continuous map $\pi_{0}(f):\pi_{0}(Y)\rightarrow\pi_{0}(X)$ which is given by $C\mapsto C'$ where $C'$ is the (unique) connected component of $Y$ containing $f(C)$. 
But the map $g$ is not necessarily continuous. For example, take $X=\mathbb{Q}$, the set of rational numbers with the discrete topology, and take $Y=\mathbb{Q}$ with the induced Euclidean topology. Then the map $f:X\rightarrow Y$ given by $x\mapsto x$ is clearly continuous. Both spaces $X$ and $Y$ are totally disconnected, i.e., the connected components of each are singletons. Hence, the map $f$ satisfies the above property. However the (bijective) map $g:\pi_{0}(Y)\rightarrow\pi_{0}(X)$ given by $C=\{x\}\mapsto f^{-1}(C)=\{x\}$ is not continuous, because the quotient topology over $\pi_{0}(X)\simeq X$ is discrete, but the quotient topology over $\pi_{0}(Y)\simeq Y$ is not discrete.
\end{remark}

Now, the structure of the connected components of many schemes is completely characterized:

\begin{theorem}\label{Theorem lrs 44} Let $X$ be a qcqs scheme and $R=\mathscr{O}_{X}(X)$. Then the connected components of $X$ are exactly of the form $f^{-1}(C)$ where $f:X\rightarrow\Spec(R)$ is the canonical morphism and $C$ is a connected component of $\Spec(R)$. 
\end{theorem}

\begin{proof} It suffices (see Remark \ref{Remark cc 3}) to show that if $C=V(I)$ is a connected component of $\Spec(R)$, then $f^{-1}(C)$ is a connected subset of $X$ where $I$ is a max-regular ideal of $R$. First, we show that $f^{-1}(C)$ is nonempty. We know that a regular ideal is generated by a set of idempotent elements, and if an ideal $I$ is generated by a subset $S$, then $V(I)=\bigcap\limits_{r\in S}V(r)$. 
If $f^{-1}(C)=\emptyset$, then by the quasi-compactness of $X$, there are finitely many idempotents $s_{1},\ldots, s_{n}\in I$ such that $X=\bigcup\limits_{i=1}^{n}X_{s_{i}}=
f^{-1}(\bigcup\limits_{i=1}^{n}D(s_{i}))$. But there exists an idempotent $s\in I$ such that $\bigcup\limits_{i=1}^{n}D(s_{i})=D(s)$, because if $e_{1}$ and $e_{2}$ are idempotents of a ring $R$ then $D(e_{1})\cup D(e_{2})=D(e_{1}+e_{2}-e_{1}e_{2})$. It follows that $X_{1}=X=X_{s}$. Then by Lemma \ref{Theorem lrs 5}, $1=s\in I$ which is a contradiction, because $I$ is a proper ideal of $R$. Therefore, $f^{-1}(C)$ is nonempty. Take $x\in f^{-1}(C)$. Then we show that $f^{-1}(C)=q(x)$, the intersection of all clopen subsets of $X$ containing the point $x$. We have $f(x)=\mathfrak{p}_{x}\in C$ and so $I\subseteq\mathfrak{p}_{x}$. It can be easily seen that the ideal $I$ is generated by $S=\{e\in\mathfrak{p}_{x}: e=e^{2}\}$, the idempotents of $\mathfrak{p}_{x}$. If $U$ is a clopen subset of $X$ containing the point $x$, then by Lemma \ref{Theorem lrs 5}, $U=X_{s}=f^{-1}(D(s))$ where $s$ is an idempotent of $R$. It follows that $1-s\in\mathfrak{p}_{x}$. This shows that the clopen subsets of $X$ containing the point $x$ are exactly of the form $X_{1-e}$ with $e\in S$. We have $C=V(I)=\bigcap\limits_{e\in S}V(e)=\bigcap\limits_{e\in S}D(1-e)$ and so $f^{-1}(C)=\bigcap\limits_{e\in S}X_{1-e}=q(x)$.    
Then by Corollary \ref{Coro iv}, $f^{-1}(C)$ is connected. This completes the proof.  
\end{proof}

Recall that by primitive idempotent of a ring $R$ we mean a nonzero idempotent $e\in R$ such that if $ee'=e'$ for some nonzero idempotent $e' \in R$, then $e=e'$. 

It can be seen that distinct primitive idempotents of a commutative ring are orthogonal.  
For a topological space $X$, the primitive idempotents of $\Clop(X)$ are exactly the connected components of $X$ that are also open subsets of $X$. The primitive idempotents of a ring $R$ and  its Boolean ring $\mathscr{B}(R)$ are the same. 

In the following result we present interesting equivalents to this concept:

\begin{lemma}\label{Lemma iv} If $e$ is an idempotent of a ring $R$ then the following assertions are equivalent: \\
$\mathbf{(i)}$ $e$ is a primitive idempotent of $R$. \\
$\mathbf{(ii)}$ $I=R(1-e)$ is a max-regular ideal of $R$. \\
$\mathbf{(iii)}$ $D(e)$ is a connected component of $\Spec(R)$. \\
$\mathbf{(iv)}$ $D(e)$ is a connected subset of $\Spec(R)$. 
\end{lemma}

\begin{proof} (i)$\Rightarrow$(ii): We have $I\neq R$, because if $I=R$ then $1=r(1-e)$ and so $e=1e=r(1-e)e=0$ which is contradiction. Hence, $I$ is a proper regular ideal of $R$. Then by Lemma \ref{Lemma ii}, it suffices to show that $R/I$ has no nontrivial idempotents. Assume $x+I$ is an idempotent for some $x\in R$. If $xe=0$ then $x=x(1-e)\in I$. Thus we may assume $xe\neq0$.
We have $x-x^2=r'(1-e)$ for some $r'\in R$. It follows that $xe=x^2e$ and so $xe$ is idempotent. But $(xe)e=xe$. Thus by hypothesis, $xe=e$. It follows that $1-x=(1-x)(1-e)\in I$. \\
(ii)$\Rightarrow$(iii): By Lemma \ref{Lemma ii}, $R/I$ has no nontrivial idempotents. Then $\Spec(R/I)$ is connected which is homeomorphic to $V(I)$. Thus $V(I)=V(1-e)=D(e)$ is a connected subset of $\Spec(R)$. Hence, $D(e)$ is contained in a connected component $C$ of $\Spec(R)$. By hypothesis, $I\neq R$ and so $e\neq0$. Thus $D(e)$ is a nonempty clopen subset of a connected space $C$. Then $D(e)=C$. \\
(iii)$\Rightarrow$(iv): There is nothing to prove. \\
(iv)$\Rightarrow$(i): By definition, every connected space is nonempty and so $e\neq0$. Suppose that $ee'=e'$ for some nonzero idempotent $e'\in R$. Then $D(e')$ is a nonempty clopen subset of $D(e)$ and so $D(e')=D(e)$. It follows that $e=e'$. 
\end{proof} 

\begin{corollary}\label{Coro lrs qcqs 21} Let $X$ be a qcqs scheme and $e$ be an idempotent of $R=\mathscr{O}_{X}(X)$. Then the following are equivalent: \\
$\mathbf{(i)}$ $e$ is a primitive idempotent of $R$. \\
$\mathbf{(ii)}$  $X_{e}$ is a connected component of $X$. \\
$\mathbf{(iii)}$ $X_{e}$ is a connected subset of $X$. 
\end{corollary}

\begin{proof} (i)$\Rightarrow$(ii): By Lemma \ref{Lemma iv}, $D(e)$ is a connected component of $\Spec(R)$. Then by Theorem \ref{Theorem lrs 44}, $f^{-1}(D(e))=X_{e}$ is a connected component of $X$ where $f: X\rightarrow\Spec(R)$ is the canonical morphism. \\
(ii)$\Rightarrow$(iii): There is nothing to prove. \\
(iii)$\Rightarrow$(i): Since $X_{e}$ is nonempty, by Lemma \ref{Theorem lrs 5}, $e$ is nonzero. If $se=s$ for some nonzero idempotent $s\in R$, then $X_{s}=X_{se}=X_{s}\cap X_{e}$ is a nonempty clopen subset of $X_{e}$ and so $X_{e}=X_{s}$. Then again by Lemma \ref{Theorem lrs 5}, $e=s$. Hence, $e$ is a primitive idempotent of $R$.
\end{proof}  

\begin{corollary} The connected components of a qcqs scheme $X$ which are also open subsets are exactly of the form $X_{s}$ where $s$ is a primitive idempotent of $R=\mathscr{O}_{X}(X)$.
\end{corollary}

\begin{proof} If $s$ is a primitive idempotent of $R$ then by Corollary \ref{Coro lrs qcqs 21}, $X_{s}$ is a connected component of $X$. Conversely, if $C$ is a connected component and open subset of $X$ then by Lemma \ref{Theorem lrs 5}, $C=X_{e}$ where $e$ is an idempotent of $R$. Thus by Corollary \ref{Coro lrs qcqs 21}, $e$ is a primitive idempotent of $R$. 
\end{proof}

\begin{remark}\label{Remark fiber 4} Let $f:X\rightarrow S$ be a morphism of schemes and $S'$ be a closed subset of $S$. Then it is well known that $f^{-1}(S')$ with the induced topology is canonically homeomorphic to the underlying space of the scheme $X\times_{S}S'$ where $S'$ is equipped with a closed subscheme structure. In particular, if $S_{1}$ and $S_{2}$ are closed subschemes of $S$ whose underlying spaces are the same (or homeomorphic), then the underlying spaces of the schemes $X\times_{S}S_{1}$ and $X\times_{S}S_{2}$ are homeomorphic (although the schemes themselves are not necessarily isomorphic).  
Also recall that a closed subset of a scheme can be equipped with a (unique) reduced induced closed subscheme structure. In particular, the reduced induced closed subscheme $f^{-1}(S')$ is canonically isomorphic to the scheme $X\times_{S}S'$ where $S'$ is equipped with the reduced induced closed subscheme structure. 
\end{remark}

\begin{corollary} Let $X$ be a qcqs scheme and $R=\mathscr{O}_{X}(X)$. Then the connected components of $X$, as reduced induced closed subschemes, are exactly of the form $X\times_{\Spec(R)}C$ where $C$ is a connected component of $\Spec(R)$ which is equipped with the reduced induced closed subscheme structure.  
\end{corollary}

\begin{proof} This follows from Theorem \ref{Theorem lrs 44} and Remark \ref{Remark fiber 4}.
\end{proof}

The following result allows us to further investigate the space of connected components:   

\begin{lemma}\label{Lemma I} Let $X$ be a topological space, $R=\Clop(X)$ and consider the map $f:\pi_{0}(X)\rightarrow\Spec(R)$ given by $[x]\mapsto M_{x}\cap R$ where $M_{x}=\mathscr{P}(X\setminus\{x\})$. Then we have: \\
$\mathbf{(i)}$ $f$ is continuous. \\
$\mathbf{(ii)}$ $f$ is injective if and only if every quasi-component of $X$ is connected. \\
$\mathbf{(iii)}$ $f$ is a surjective and closed map provided that $X$ is quasi-compact. 
\end{lemma}

\begin{proof} (i): For each $x\in X$, then $M_x$ is a maximal ideal of $\mathscr{P}(X)$ and so $M_x\cap R$ is a prime (maximal) ideal of $R$. If $[x]=[y]$ for some $x,y\in X$ then we need to show that $M_x\cap R=M_y\cap R$. Take $A\in M_x\cap R$. If $A\notin M_y\cap R$ then $A$ is a clopen subset of $X$ containing $y$ and so 
the connected component $[y]$ is contained in $A$. It follows that $x\in A$ which is a contradiction. Thus $M_x\cap R\subseteq M_y\cap R$ and so $M_x\cap R= M_y\cap R$, because $M_x\cap R$ is a maximal ideal of $R$. 
Hence, the above map $f$ is well-defined. This map is continuous, because for each $A\in R$ then $g^{-1}(f^{-1}(D(A)))=A$ is an open subset of $X$ where $D(A)=\{P\in\Spec(R): A\notin P\}$ is a basis open of $\Spec(R)$ and $g:X\rightarrow \pi_{0}(X)$ 
is the canonical map that is given by $x\mapsto[x]$. \\ 
(ii): We know that every connected component $[x]$ of $X$ is contained in the intersection of all clopen subsets of $X$ containing the point $x\in X$. Assume $f$ is injective. If $y\in\bigcap\limits_{\substack{A\in\Clop(X),
\\x\in A}}A$ then $M_y\cap R\subseteq M_x\cap R$ and so $M_y\cap R=M_x\cap R$, because $M_y\cap R$ is a maximal ideal of $R$. But $f$ is injective and so $y\in [x]$. Conversely, suppose $M_x\cap R=M_y\cap R$ for some $x,y\in X$. If $y\notin[x]$ then by hypothesis, there exists some clopen $A$ in $X$ containing $x$ such that $y\notin A$. This yields that $A\in M_y\cap R$ which is a contradiction. Thus $y\in[x]$ and so $[x]=[y]$. Hence, $f$ is injective.  \\
(iii): If $X$ is quasi-compact then $\pi_{0}(X)$ is quasi-compact. The space $\Spec(R)$ is Hausdorff, because the prime spectrum of a commutative ring $R$ is Hausdorff if and only if every prime ideal of $R$ is maximal. It can be seen that every continuous map from a quasi-compact space into a Hausdorff space is a closed map. Hence, $f$ is a closed map. To see surjectivity, suppose that $P$ is a prime (maximal) ideal of $R$ with $P\neq M_x\cap R$ for all $x\in X$. 
Then for each $x\in X$, $P$ is not contained in $M_x\cap R$  and so there exists some $A_x\in P$ such that $x\in A_x$. Thus $\{A_x: x\in X\}$ is an open covering of $X$. Since $X$ is quasi-compact, there exist finitely many points $x_{1},\ldots, x_{n}\in X$ with $X=\bigcup\limits_{k=1}^{n}A_{x_{k}}\in P$. But this is a contradiction, because $X$ is the unit element of the ring $R$.  
\end{proof}  

In the above lemma, the converse of (iii) does not hold. For example, take a connected space $X$ that is not quasi-compact, then $\pi_{0}(X)$ and $\Spec(R)$ will be singletons and so the map $f:\pi_{0}(X)\rightarrow\Spec(R)$ will be surjective and closed. 

The counterexample given in Remark \ref{Remark cc 3} shows that if a continuous map of topological spaces $f: X\rightarrow Y$ has the property that $f^{-1}(C)$ is connected for all connected components $C$ in $Y$, then the continuous (bijective) map  $\pi_{0}(f): \pi_{0}(X)\rightarrow \pi_{0}(Y)$ is not a homeomorphism in general. However, in a geometric case, it is a homeomorphism:  

\begin{corollary}\label{Coro qcqs cc 1} Let $X$ be a qcqs scheme and $R=\mathscr{O}_{X}(X)$. Then 
we have the following canonical homeomorphisms: 
$$\pi_{0}(X)\simeq\Spec(\mathscr{B}(R))\simeq
\pi_{0}(\Spec(R)).$$ 
\end{corollary}

\begin{proof} By Lemma \ref{Theorem lrs 5}, the map $\mathscr{B}(R)\rightarrow\Clop(X)$ given by $s\mapsto X_{s}$ is an isomorphism of rings. Hence, their prime spectra are homeomorphic $\Spec(\mathscr{B}(R))\simeq\Spec(\Clop(X))$. By Corollary \ref{Coro iv}, every quasi-component of $X$ is connected. Then by Lemma \ref{Lemma I}, $\pi_{0}(X)\simeq
\Spec(\mathscr{B}(R))$. In particular, for any ring $R$, we have the canonical homeomorphism $\pi_{0}(\Spec(R))\simeq
\Spec(\mathscr{B}(R))$. This completes the proof. 
\end{proof}

\begin{example} We first show with an example that there is a scheme $X$ such that $\pi_{0}(X)$, the space of connected components, is not necessarily homeomorphic to $\pi_{0}(\Spec(R))$
where $R=\mathscr{O}_{X}(X)$. For instance, consider the scheme $X=\coprod\limits_{I}\Spec(k)$, the disjoint union of copies of $\Spec(k)$ where $k$ is a field, and the index set $I$ is infinite. Then each point of $X$ is a connected component, and so $\pi_{0}(X)=I$ with the discrete topology. But $R=\mathscr{O}_{X}(X)=\prod\limits_{I}k$ and so $\Spec(R)$ is the Stone-Čech compactification of the infinite discrete space $I$ and so has the cardinality $2^{2^{|I|}}$. But $\Spec(R)$ is totally disconnected. In fact, every zero-dimensional affine scheme and more generally every quasi-separated zero-dimensional scheme is totally disconnected (see \cite[Tag 0CKV]{Johan}). Then $\pi_{0}(\Spec(R))=\Spec(R)$ has many points than $\pi_{0}(X)$, the two spaces are not even in bijection, let alone homeomorphic. 

Next, consider the natural morphism $f:X\rightarrow \Spec(R)$. Then for each point $x=(0,i)\in X$ with $0\in X_{i}=\Spec(k)$ we have $f(x)=\pi_{i}^{-1}(0)$ where $\pi_{i}:R=\prod\limits_{I}k\rightarrow k_{i}=k$ is the projection map. Since the index set $I$ is infinite, there are (so many) prime ideals in $R$ which are not of the form $\pi_{i}^{-1}(0)$ for some $i$. This shows that if $P$ is  such a prime ideal then $f^{-1}(P)$ is empty. Thus we also find an example of a scheme $X$ such that $\Spec(R)$ has a connected component $C$ for which $f^{-1}(C)$ is not connected where $f:X\rightarrow\Spec(R)$ is the natural morphism. 
\end{example}

There are also examples of schemes $X$ such that $\pi_{0}(X)$ and $\pi_{0}(\Spec(R))$ are homeomorphic where $R=\mathscr{O}_{X}(X)$, but $X$ fails to be quasi-compact or quasi-separated.

\begin{corollary}\label{Coro III} If $X$ is a compact space or a quasi-spectral space, then $\pi_{0}(X)\simeq\Spec(\Clop(X))$.
\end{corollary}

\begin{proof} Every quasi-component of $X$ is connected (apply \cite[Theorem 6.1.23]{Engelking} in the compact case, and Theorem \ref{Theorem I T-D-F}
in the quasi-spectral case). Then apply Lemma \ref{Lemma I}.
\end{proof}

\begin{corollary}\label{Coro v-5} Let $X$ be a topological space. Then $X$ is compact totally disconnected if and only if $X\simeq\Spec(\Clop(X))$.
\end{corollary}

\begin{proof} The implication ``$\Rightarrow$" follows from Corollary \ref{Coro III}. The reverse implication follows from the fact that the prime spectrum of a zero dimensional ring is compact totally disconnected. 
\end{proof}

Both the new theorem (every quasi-component of a quasi-spectral space is connected) and the classical theorem (every quasi-component of a compact space is connected) are supplemented by the following subtle examples showing that there are quasi-compact spaces that are neither Hausdorff nor quasi-spectral, but whose quasi-components can be connected or disconnected.   

\begin{example} We give an example of a quasi-compact space that is neither Hausdorff nor quasi-spectral, and has a quasi-component that is not connected. To this end, let $X=\{1/n : n\geqslant1\}\cup\{0,b\}$ with $b$ is a (fixed) rational number other than 0 and all $1/n$. The topology over $X$ is consisting of all $U$ such that $U$ is a subset of $S=\{1/n : n\geqslant1\}$ ($S$ is considered discrete) or if $0 \in U$ (resp. if $b \in U$) then $U$ contains almost all of the $1/n$, i.e., $S\setminus U$ is a finite set. Then with this topology, the sequence $(1/n)$ is convergent to both $0$ and $b$. Hence, the space $X$ is quasi-compact, but it is not Hausdorff. Moreover, it can be seen that the intersection of all clopen subsets of $X$ containing the point $0$ equals $\{0,b\}$. But $\{0,b\}$ is not connected, because both $S\cup\{0\}$ and $S\cup\{b\}$ are (quasi-compact) open subsets of $X$. Also, $X$ is not quasi-spectral, because $(S\cup\{0\})\cap(S\cup\{b\})=S$ is not quasi-compact.  
\end{example}

\begin{example} Here we give an example of a quasi-compact space that is neither Hausdorff nor quasi-spectral, but every quasi-component is connected. To achieve this purpose, take $X=\{-1/n : n\geqslant1\}\cup[0,1]$, the union of $S=\{-1/n : n\geqslant1\}$ and the closed interval $[0,1]$. Similarly above, the topology over $X$ is consisting of all $U$ such that $U$ is a subset of $S$ or if $x\in U$ for some $x\in[0,1]$ then $U$ is the union of an open neighbourhood of $x$ in $[0,1]$ and the complement of a finite subset of $S$. Then with this topology, every element of $[0,1]$ is a limit point of the sequence $S$. Hence, the space $X$ is quasi-compact, but it is not Hausdorff. Moreover, every quasi-component of $X$ is connected, because the quasi-component of each $x\in S$ equals the connected subset $\{x\}$, and the quasi-component of 0 equals $[0,1]$ which is connected. Finally, $X$ is not quasi-spectral. Because every closed subspace of a quasi-spectral space is quasi-spectral, but the closed subspace $[0,1]$ of $X$ is not quasi-spectral. Indeed, a connected Hausdorff space with nontrivial topology cannot be a quasi-spectral space. 
\end{example}

\section{Connected components of projective spaces}

To prove the main result of this section, we need the following lemmas: 

\begin{lemma}\label{Lemma Proj 5} Let $f:X\rightarrow Y$ be a continuous map of topological spaces such that all fibers of $f$ are connected, and one of the following conditions hold:  \\
$1)$ $f$ is a closed map. \\
$2)$ $f$ is an open map. \\
$3)$ $f$ has a section. \\
Then we have: \\
$\mathbf{(i)}$ $X$ is connected if and only if $Y$ is connected. \\
$\mathbf{(ii)}$ A subset $C$ of $Y$ is connected if and only if $f^{-1}(C)$ is a connected subset of $X$. 
\end{lemma}

\begin{proof} (i): The map $f$ is surjective, since all fibers of $f$ are connected and hence they are nonempty. Thus if $X$ is connected then $Y=f(X)$ is connected as well. Conversely, assume that $Y$ is connected. 
Since $f$ is surjective, $X$ is nonempty. 
Assume that $f$ is a closed map. Let $U$ be a nonempty clopen of $X$. Then $f(U)$ and $f(U^{c})$ are closed subsets of $Y$ where $U^{c}=X\setminus U$. It can be easily seen that $f(U)\cap f(U^{c})=\emptyset$, because the fibers of $f$ are connected. This shows that $f(U)$ is a nonempty clopen subset of $Y$ and $U=f^{-1}(f(U))$. Since $Y$ is connected, $f(U)=Y$ and so $U=X$. Hence, $X$ is connected in this case. The open map case is proven similarly. Assume $f$ has a section, i.e., there is a continuous map $g:Y\rightarrow X$ such that $fg:Y\rightarrow Y$ is the identity map. Then $g(Y)$ is a connected subset of $X$.  Let $U$ be a clopen of $X$. We have $g(Y)\subseteq U$ or $g(Y)\subseteq U^{c}$. Suppose $g(Y)\subseteq U$. Then all fibers of $f$ meet $U$, because if $y\in Y$ then  $g(y)\in f^{-1}(y)$. 
Thus all fibers of $f$ are contained in $U$, since they are connected. 
This shows that $U^{c}$ is empty, i.e., $U=X$. If $g(Y)\subseteq U^{c}$ then $U$ will be empty. Thus $X$ is also connected in this case. \\
(ii): Apply (i) for the induced map $g:f^{-1}(C)\rightarrow C$ which is given by $x\mapsto f(x)$.     
\end{proof}

Recall that if $I$ is a graded ideal of an $\mathbb{N}$-graded ring $R$, then by $V_{+}(I)$ we mean $V(I)\cap\Proj(R)$.

Also recall that if $Y$ is a subspace of a topological space $X$ and $y\in Y$ then the closure of $\{y\}$ in $Y$ equals to $\overline{\{y\}}\cap Y$ where $\overline{\{y\}}$ is the closure of $\{y\}$ in $X$.
This topological observation leads us to the following geometric conclusion which, in particular,  characterizes the irreducible components of projective spaces:

\begin{lemma}\label{Lemma very nice 89} Let $R$ be an $\mathbb{N}$-graded ring. Then the irreducible closed subsets of $\Proj(R)$ are exactly of the form $V_{+}(\mathfrak{p})$ with $\mathfrak{p}\in\Proj(R)$. In addition, the irreducible components of $\Proj(R)$ are exactly of the form $V_{+}(\mathfrak{p})$ with $\mathfrak{p}\in\Min(R)\cap\Proj(R)$.
\end{lemma}

\begin{proof} Let $Z$ be an irreducible closed subset of $\Proj(R)$. We know that the underlying space of every scheme is a sober space, i.e., every irreducible closed subset has a (unique) generic point. So there exists a point $\mathfrak{p}\in\Proj(R)$ such that $Z=c\ell(\{\mathfrak{p}\})$ where $c\ell(\{\mathfrak{p}\})$ is the closure of $\{\mathfrak{p}\}$ in $\Proj(R)$. But $\Proj(R)$ is a subspace of $\Spec(R)$. Then $Z=c\ell(\{\mathfrak{p}\})=
\overline{\{\mathfrak{p}\}}\cap\Proj(R)=
V(\mathfrak{p})\cap \Proj(R)=V_{+}(\mathfrak{p})$ where $\overline{\{\mathfrak{p}\}}=V(\mathfrak{p})$ is the closure of $\{\mathfrak{p}\}$ in $\Spec(R)$. \\ 
This argument also shows that $V_{+}(\mathfrak{p})$ is an irreducible closed subset of $\Proj(R)$ for all $\mathfrak{p}\in\Proj(R)$. \\
If $Z$ is an irreducible component of $\Proj(R)$ then $Z$ is an irreducible closed subset and so $Z=V_{+}(\mathfrak{p})$ for some $\mathfrak{p}\in\Proj(R)$. Then we claim that $\mathfrak{p}$ is a minimal prime of $R$. Indeed, there exists a minimal prime $\mathfrak{q}$ of $R$ with $\mathfrak{q}\subseteq\mathfrak{p}$. But every minimal prime of an $\mathbb{N}$-graded ring is a graded ideal. 
It follows that $\mathfrak{q}\in\Proj(R)$. Then $Z=V_{+}(\mathfrak{p})\subseteq V_{+}(\mathfrak{q})$ and so $V_{+}(\mathfrak{p})= V_{+}(\mathfrak{q})$, because $Z$ is an irreducible component. It follows that $\mathfrak{q}\in V_{+}(\mathfrak{p})$ and so $\mathfrak{p}=\mathfrak{q}$. \\
The above argument also show that for any minimal prime $\mathfrak{p}$ of $R$ with $\mathfrak{p}\in\Proj(R)$, then $V_{+}(\mathfrak{p})$ is an irreducible component of $\Proj(R)$. This completes the proof.
\end{proof}

\begin{remark} Regarding the above lemma, let $\mathfrak{p}$ be a minimal prime of an $\mathbb{N}$-graded ring $R=\bigoplus\limits_{n\geqslant0}R_{n}$. Then $\mathfrak{p}\notin\Proj(R)$ if and only if $\mathfrak{p}=\mathfrak{p}_{0}\oplus R_{+}$ where $\mathfrak{p}_{0}=R_{0}\cap\mathfrak{p}$ is a minimal prime of $R_{0}$. As a specific example, consider the 
$\mathbb{N}$-graded ring $R=k[x,y]/(xy)$ with $\deg(x)=0$ and $\deg(y)=1$. Then the minimal prime $P=(y)/(xy)$ of $R$ is not a member of $\Proj(R)$. \\
If $I$ is a graded ideal of an $\mathbb{N}$-graded ring $R$, then we have $V_{+}(I)=V_{+}(I\cap R_{+})$ where $R_{+}=\bigoplus\limits_{n\geqslant1}R_{n}$ is the irrelevant ideal of $R$. 
It can be also seen that the closed subsets of $\Proj(R)$ are exactly of the form $V_{+}(I)$ where $I$ is a unique graded radical ideal of $R$ contained in $R_{+}$. 
\end{remark}

\begin{corollary}\label{Lemma proj 77} Let $R$ be an $\mathbb{N}$-graded ring. \\
$\mathbf{(i)}$ If $R$ has a unique minimal prime $\mathfrak{p}$, then $\Proj(R)$ is irreducible if and only if $\mathfrak{p}\in\Proj(R)$, or equivalently, $R_{+}$ is not contained in $\mathfrak{p}$. \\
$\mathbf{(ii)}$ If $R$ is an integral domain, then $\Proj(R)$ is irreducible if and only if $R_{+}\neq0$. 
\end{corollary}

\begin{proof} The statement (i) is a special case of Lemma \ref{Lemma very nice 89}, and   
(ii) follows from (i).     
\end{proof}

\begin{corollary}\label{Coro special case 1} If $R$ is an integral domain, then $\mathbb{P}_{R}^{n}=\Proj(R[x_{0},\ldots,x_{n}])$ is irreducible. 
\end{corollary}

\begin{proof} This is a special case of Corollary \ref{Lemma proj 77}(ii). 
\end{proof}

\begin{theorem}\label{Theorem Proj} For any scheme $S$, the connected components of $\mathbb{P}_{S}^{n}=
\mathbb{P}_{\mathbb{Z}}^{n}\times_{\Spec(\mathbb{Z})}S$ are exactly of the form $f^{-1}(C)$
where  $f:\mathbb{P}_{S}^{n}\rightarrow S$ is the canonical morphism and $C$ is a connected component of $S$. 
\end{theorem}

\begin{proof} To prove the assertion, we will apply Lemma \ref{Lemma Proj 5}. First, we show that 
all fibers of $f:\mathbb{P}_{S}^{n}\rightarrow S$ are connected. If $s\in S$ then $f^{-1}(s)$ with the induced topology is homeomorphic to the underlying space of the scheme $\mathbb{P}_{S}^{n}\times_{S}\Spec(\kappa(s))\simeq
\mathbb{P}^{n}_{\mathbb{Z}}\times_{\Spec(\mathbb{Z})}
\Spec(\kappa(s))=
\mathbb{P}_{\kappa(s)}^{n}$ where $\kappa(s)$ is the residue field of $S$ at the point $s$. But for any field $k$, by Corollary \ref{Coro special case 1}, the projective space $\mathbb{P}_{k}^{n}=\Proj k[x_{0},\ldots,x_{n}]$ is connected, even irreducible. 
Therefore, the fiber $f^{-1}(s)$ is connected.
It is well known that $f$ is a closed map, even universally closed (see e.g. \cite[Tag 01NH]{Johan}). This completes the proof.  
\end{proof} 

As we observed in Lemma \ref{Lemma Proj 5}, the above theorem holds more generally for any
morphism of schemes that satisfies the conditions of this lemma. For instance, a proper morphism of schemes with connected fibers is a typical example of such morphisms.

\begin{corollary}\label{Lemma Proj 6} A scheme $S$ is connected if and only if $\mathbb{P}_{S}^{n}$ is connected. 
\end{corollary}

\begin{proof} This follows from Theorem \ref{Theorem Proj}.
\end{proof}

\begin{corollary}\label{Coro 10 proj} For any ring $R$, the connected components of $\mathbb{P}_{R}^{n}=\Proj R[x_{0},\ldots,x_{n}]$ are exactly of the form $f^{-1}(C)$ where $f:\mathbb{P}_{R}^{n}\rightarrow\Spec(R)$ is the canonical morphism and $C$ is a connected component of $\Spec(R)$. 
\end{corollary}

\begin{proof} This follows from Theorem \ref{Theorem Proj}.
\end{proof}

\begin{corollary} For any scheme $S$, the connected components of $\mathbb{P}_{S}^{n}$, as reduced induced closed subschemes, are exactly of the form $\mathbb{P}_{C}^{n}$
where $C$ is a connected component of $S$ which is equipped with the reduced induced closed subscheme structure.
\end{corollary}

\begin{proof} This follows from Theorem \ref{Theorem Proj} and Remark \ref{Remark fiber 4}.
\end{proof}

Recall that if $R=\bigoplus\limits_{n\geqslant0}R_{n}$ is an $\mathbb{N}$-graded ring and $S$ is an $R_{0}$-algebra, then the ring $R\otimes_{R_{0}}S$ is $\mathbb{N}$-graded with homogeneous components $R_{n}\otimes_{R_{0}}S$ and we have the canonical isomorphism of schemes $\Proj(R)\times_{\Spec(R_{0})}\Spec(S)
\simeq\Proj(R\otimes_{R_{0}}S)$. Using this and the above result, we arrive at the following conclusion:  

\begin{corollary} For any ring $R$, the connected components of $\mathbb{P}_{R}^{n}$, as reduced induced closed subschemes, are exactly of the form $\mathbb{P}_{R/I}^{n}$ where $I$ is a radical ideal of $R$ such that $V(I)$ is a connected component of $\Spec(R)$. 
\end{corollary}

\begin{remark} Inspired by the above results (Theorem \ref{Theorem Proj} and Corollary \ref{Coro 10 proj}), the first approach that comes to mind is that, for an $\mathbb{N}$-graded ring $R=\bigoplus\limits_{n\geqslant0}R_{n}$, 
the connected components of scheme $\Proj(R)$ are probably of the form $\Proj(R\otimes_{R_{0}}R_{0}/I)$ where $I$ is an ideal of $R_{0}$ such that $V(I)$ is a connected component of $\Spec(R_{0})$. But then we observed that this claim is not true. As a counterexample, consider the $\mathbb{N}$-graded ring $R=k[x,y]/(x^{2}-y^{2})$ with $R_{0}=k$ is a field of characteristic $\neq2$, then $\Proj(R)$ is not connected, since it can be written as the disjoint union of nonempty opens $D_{+}(f)$ and $D_{+}(g)$ where $f$ and $g$ are the images of $x-y$ and $x+y$ in $R$. 
\end{remark}

We conclude this section with the following auxiliary result:

\begin{proposition} If $R=\bigoplus\limits_{n\geqslant0}R_{n}$ is an $\mathbb{N}$-graded ring, then the topology over $\Proj(R)$ with the basis opens $D_{+}(f)$ where $f\in R$ is a homogeneous element of positive degree, and the subspace topology over $\Proj(R)$ induced by $\Spec(R)$ are the same.
\end{proposition}

\begin{proof} It is clear that the topology over $\Proj(R)$ with the basis opens $D_{+}(f)$ is contained in the subspace topology. It is also clear that the basis opens of the subspace topology over $\Proj(R)$ are of the form $D(r)\cap\Proj(R)$ where $r\in R$. If $r$ is a homogeneous element of degree zero, then $D(r)\cap\Proj(R)=\bigcup\limits_{f} D_{+}(rf)$ where the union is taken over the set of homogeneous elements of positive degree $f\in R$ such that $rf\neq0$. This shows that in this case, $D(r)\cap\Proj(R)$ is an open subset in the second topology, i.e., the topology with the basis opens $D_{+}(f)$. \\ 
In general, we may write $r=\sum\limits_{n\geqslant0}r_{n}$ where $r_{n}\in R_{n}$ for all $n\geqslant0$ (the $r_{n}=0$ except for finitely many indices $n$). Then $D(r)\cap\Proj(R)=
\bigcup\limits_{n\geqslant0}D(r_{n})\cap\Proj(R)$.
If $r_{0}\neq0$, then in the above we observed that $D(r_{0})\cap\Proj(R)$ is an open subset of the second topology. If $r_{n}\neq0$ for some $n\geqslant1$, then $r_{n}$ is a homogeneous element of positive degree and so $D_{+}(r_{n})=D(r_{n})\cap\Proj(R)$ is a basis open of the second topology. Note that if $r_{n}=0$ for some $n\geqslant0$, then $D(r_{n})=\emptyset$. Therefore, $D(r)\cap\Proj(R)$ is an open subset of the second topology. 
\end{proof}

\section{Finiteness of connected components}

The main results of this section give us very useful equivalences for the finiteness of the number of connected components of an arbitrary (quasi-compact) topological space in terms of purely algebraic conditions.

First recall that finite nonzero Boolean rings are precisely of the form $(\mathbb{Z}/2)^{n}=
\prod\limits_{k=1}^{n}\mathbb{Z}/2$ where $n\geqslant1$ is a natural number and $\mathbb{Z}/2=\{0,1\}$ is the field of integers modulo 2. Indeed, if $R$ is a finite Boolean ring then $X=\Spec(R)$ is a finite discrete space and so we have isomorphisms of rings $R\simeq\Clop(X)=\mathscr{P}(X)\simeq(\mathbb{Z}/2)^{n}$ where $n=|X|$ is the number of primitive idempotents of $R$.

If $X$ is a topological space, then by $H_{0}(X)$ we mean the set of all continuous functions $X\rightarrow\mathbb{Z}$ where the set of integers $\mathbb{Z}$ is equipped with the discrete topology. This set by the addition $(f+g)(x)=f(x)+g(x)$ and multiplication $(f\cdot g)(x)=f(x)g(x)$ is a (commutative) ring.   
If $f:X\rightarrow\mathbb{Z}$ is a continuous map and $C$ is a connected component of $X$ then $f(C)$ is a connected subset of $\mathbb{Z}$ and so $f(C)=\{r_C\}$ for some integer $r_C\in\mathbb{Z}$. Then the map $H_{0}(X)\rightarrow\mathbb{Z}^{\kappa}=
\prod\limits_{\kappa}\mathbb{Z}$ given by $f\mapsto(r_C)_{C\in\pi_{0}(X)}$
is an injective morphism of rings with $\kappa=|\pi_{0}(X)|$. 
 
For any topological space $X$, the map $f:\Clop(X)\rightarrow\mathscr{B}\big(H_{0}(X)\big)$ given by $A\mapsto f_A$ is an isomorphism of rings where $f_A:X\rightarrow\mathbb{Z}$ is the characteristic function of $A$ that is defined by $f_A(x)=1$ if $x\in A$ otherwise $f_A(x)=0$.

We also need the following simple observation in the first main result of this section:  

\begin{remark}\label{Remark bir-i} Let $f:A\rightarrow B$ be a morphism of rings such that $B$ as an $A$-module is generated by a subset $S\subseteq B$ with $bb'=0$ for all distinct elements $b,b'\in S$. Then the unit element of $B$ can be written as $1=\sum\limits_{k=1}^{n}b_{k}$ where $b_{k}\in S$ for all $k$. Indeed, there exists finitely many (distinct) elements $b_1,\ldots,b_n\in S$ such that $1=\sum\limits_{k=1}^{n}f(a_{k}) b_k$ with $a_{k}\in A$ for all $k$. But we have $b_{i}=\big(\sum\limits_{k=1}^{n}f(a_{k})b_k\big)b_{i}=
f(a_{i})b_{i}$. Hence, $1=\sum\limits_{k=1}^{n}b_k$. In particular, the unit ideal of $B$ is generated by the $b_k$, i.e., $B=(b_1,\ldots,b_n)$. 
\end{remark} 

Now we prove the first main result of this section:

\begin{theorem}\label{Theorem iii-uch} For a topological space $X$ the following assertions are equivalent: \\
$\mathbf{(i)}$ $X$ has finitely many connected components. \\
$\mathbf{(ii)}$ There exists an isomorphism of rings $\Clop(X)\simeq(\mathbb{Z}/2)^{n}$ for some natural number $n\geqslant0$. \\
$\mathbf{(iii)}$ The unit ideal of $\Clop(X)$ is generated by a set of its primitive idempotents. \\
$\mathbf{(iv)}$  $\Clop(X)$ as a vector space  over the field $\mathbb{Z}/2$ has a basis 
consisting of primitive idempotents. \\
$\mathbf{(v)}$ There exists an isomorphism of rings $H_{0}(X)\simeq\mathbb{Z}^{n}$ for some natural number $n\geqslant0$. \\
$\mathbf{(vi)}$ $H_{0}(X)$ is a free Abelian group with a basis consisting of primitive idempotents. \\
$\mathbf{(vii)}$ The unit ideal of $H_{0}(X)$ is generated by a set of its primitive idempotents.
\end{theorem}

\begin{proof} (i)$\Rightarrow$(ii), (iii), (iv) and (v): Without loss of generality, we may assume $X$ is nonempty (because if $X$ is empty then $\Clop(X)\simeq(\mathbb{Z}/2)^{0}$ and $H_{0}(X)\simeq\mathbb{Z}^{0}$ are the zero rings, and the unit ideal of $\Clop(X)$ is generated by the empty set). Assume $C_1,\ldots,C_n$ are all (distinct) connected components of $X$ with $n\geqslant1$. Then $C_k$ is a clopen subset of $X$ and so $C_k\in\Clop(X)$ for all $k$. In fact, it can be seen that each $C_k$ is a primitive idempotent of $\Clop(X)$.
The unit element of $\Clop(X)$ can be written as $X=\bigcup\limits_{k=1}^{n}C_k=\sum\limits_{k=1}^{n}C_k$. This shows that the unit ideal of $\Clop(X)$ is generated by the $C_k$, i.e., $\Clop(X)=(C_1,\ldots,C_n)$. 
If $A\in\Clop(X)$ then $A=\bigcup\limits_{k=1}^{n}A\cap C_k=\sum\limits_{k=1}^{n}AC_k$. But $AC_k=A\cap C_k$ is a clopen subset of the connected space $C_k$. Thus $A\cap C_k$ is the empty or the whole space $C_k$. Then we may write $AC_k=s_{k}\cap C_k=s_{k}C_k$ with $s_k\in\{\emptyset, X\}$. So $A=\sum\limits_{k=1}^{n}s_{k}C_k$. Then we show that such a presentation is unique. If $\sum\limits_{k=1}^{n}s_{k}C_{k}=
\emptyset$ then $s_{i}\cap C_{i}=s_{i}C_{i}=C_{i}(\sum\limits_{k=1}^{n}s_{k}C_{k})=
\emptyset$ and $C_i\neq\emptyset$. Thus $s_i=\emptyset$ for all $i$. This shows that $\{C_1,\ldots,C_n\}$ is a basis for $\mathbb{Z}/2$-space $\Clop(X)$. It also shows that the Boolean ring $\Clop(X)$ has $2^n$ elements and so it is isomorphic to $(\mathbb{Z}/2)^{n}$. Finally, we have an injective morphism of rings $H_{0}(X)\rightarrow\mathbb{Z}^{n}$ that is given by $f\mapsto(r_1,\ldots,r_n)$ with $f(C_k)=\{r_k\}$ for all $k$. If $(a_1,\ldots,a_n)\in\mathbb{Z}^{n}$, then the map $g:\Spec(R)\rightarrow\mathbb{Z}$
defined by $g(C_k)=\{a_k\}$ for all $k$ is continuous. Because for each $b\in\mathbb{Z}$, we have $g^{-1}(\{b\})=\bigcup\limits_{a_k=b}C_k$ is an open subset of $X$. Hence, the above ring map $H_{0}(X)\rightarrow\mathbb{Z}^{n}$ 
is an isomorphism. \\   
(ii)$\Rightarrow$(i): It can be seen that the primitive idempotents of a direct product of nonzero rings $\prod\limits_{k\in S}R_k$ are precisely of the form $e_k=(\delta_{i,k})_{i\in S}$ with $k\in S$ and $\delta_{i,k}$ is the Kronecker delta. By hypothesis, there exists an isomorphism of rings $f:(\mathbb{Z}/2)^{n}\rightarrow\Clop(X)$. In the ring $(\mathbb{Z}/2)^{n}$ we have $1=\sum\limits_{k=1}^{n}e_k$ and so $X=\sum\limits_{k=1}^{n}A_k=\bigcup\limits_{k=1}^{n}A_k$ with $A_k=f(e_k)$ is a primitive idempotent of $\Clop(X)$ for all $k$. Thus each $A_k$ is a connected component of $X$. Conversely, assume $C$ is a connected component of $X$. We may write $C=\bigcup\limits_{k=1}^{n}C\cap A_k$. It follows that $C\cap A_k\neq\emptyset$ for some $k$. 
But $C\cap A_k$ is a clopen subset of $C$. Thus $C\cap A_k=C$. This yields that $C\subseteq A_k$ and so $C=A_k$. Hence, $\pi_{0}(X)=\{A_1,\ldots,A_n\}$. \\
(iii)$\Rightarrow$(i): We may assume $\Clop(X)\neq0$ (because $X$ is the empty if and only if $\Clop(X)=0$).
Then by hypothesis, there exist finitely many (distinct) primitive idempotents $A_1,\ldots,A_n\in\Clop(X)$  with $n\geqslant1$ such that
$X=\sum\limits_{k=1}^{n}B_{k}A_{k}=
\bigcup\limits_{k=1}^{n}B_{k}\cap A_{k}$ where $B_{k}\in\Clop(X)$ for all $k$. It is clear that $\{A_1,\ldots,A_n\}\subseteq\pi_{0}(X)$. Conversely, take $C\in\pi_{0}(X)$. We have $C=\bigcup\limits_{k=1}^{n}C\cap B_{k}\cap A_{k}$. Thus $C\cap B_{k}\cap A_{k}\neq\emptyset$ for some $k$. But $C\cap B_{k}\cap A_{k}$ is a clopen subset of the connected space $C$. Then $C=C\cap B_{k}\cap A_{k}\subseteq A_k$ and so $C=A_k$. Hence, $\pi_{0}(X)=\{A_1,\ldots,A_n\}$. \\
(iv)$\Rightarrow$(iii): It is clear (see Remark \ref{Remark bir-i}). \\
(v)$\Rightarrow$(vi): By hypothesis, there exists an isomorphism of rings $\phi:\mathbb{Z}^{n}\rightarrow H_{0}(X)$. Consider the primitive idempotents  $e_k=(\delta_{i,k})_{i=1}^{n}$ of $\mathbb{Z}^n$, then we will show that $\{f_1,\ldots,f_n\}$ is a basis for the $\mathbb{Z}$-module $H_{0}(X)$ where each $f_k=\phi(e_k)$ is a primitive idempotent of $H_{0}(X)$. If $g\in H_{0}(X)$ then there exists some $(a_1,\ldots,a_n)\in\mathbb{Z}^{n}$ such that $g=\phi\big((a_1,\ldots,a_n)\big)=
\phi(\sum\limits_{k=1}^{n}a_{k}e_k)=
\sum\limits_{k=1}^{n}\phi(a_{k}e_k)=
\sum\limits_{k=1}^{n}a_{k}f_k$. To complete the proof, we show that such a representation is unique. 
If $\sum\limits_{k=1}^{n}a_{k}f_k=0$ then $f_{i}(\sum\limits_{k=1}^{n}a_{k}f_k)=0$. It follows that $a_{i}f_{i}=0$. But $f_i\neq0$ and so $f_{i}(x)$ is a nonzero integer for some $x\in X$. We have $a_{i}f_{i}(x)=0$ and so $a_i=0$ for all $i$. \\
(vi)$\Rightarrow$(vii): It is clear (see Remark \ref{Remark bir-i}). \\
(vii)$\Rightarrow$(iii): The assertion is deduced from the identification $\mathscr{B}\big(H_{0}(X)\big)\simeq\Clop(X)$ and the fact that distinct primitive idempotents $e$ and $e'$ are orthogonal and hence $e\oplus e'=e+e'$.   
\end{proof} 

\begin{corollary}\label{Coro 8-sekiz} A topological space $X$ has finitely many connected components  if and only if $\Clop(X)$ is finite. In this case, $|\Clop(X)|=2^n$ with $n=|\pi_{0}(X)|$ is the number of primitive idempotents of $\Clop(X)$. 
\end{corollary}

\begin{proof} It follows from Theorem \ref{Theorem iii-uch} (and its proof).
\end{proof}

If a topological space $X$ has finitely many irreducible components, then $X$ has finitely many connected components. But its converse is not true. As an example, consider $\Spec(R)$ where $R$ is a local ring with infinitely many minimal primes. 

\begin{example} For any topological space $X$, the cardinal of primitive idempotents of $\Clop(X)\leqslant|\pi_{0}(X)|$. If $\Clop(X)$ is finite, then the equality holds (see Corollary \ref{Coro 8-sekiz}). But by giving an example, we show that the converse is not true. \\
Let $S$ be an infinite set and let $R$ be the Boolean ring of all subsets of $S$ that are finite or cofinite (its complement is finite). Then it can be seen that the prime (maximal) ideals of $R$ are precisely $\Fin(S)$ or of the form $M_{x}\cap R$ where $x\in S$ and $M_{x}=\mathscr{P}(S\setminus\{x\})$, and so $X=\Spec(R)$ has precisely $|S|+1=|S|$ points (in fact, it can be shown that $X$ is the one-point compactification of the discrete space $S$ whose point at infinity is $\Fin(S)$, the set of all finite subsets of $S$). The prime spectrum of every zero dimensional ring is totally disconnected, and hence the cardinal of connected components of $X$ is equal to $|S|$. The cardinal of primitive idempotents of $\Clop(X)\simeq R$ also equals $|S|$, because the primitive idempotents of $R$ are precisely the singleton subsets of $S$. But $\Clop(X)$ is infinite.
\end{example}

\begin{example} In Corollary \ref{Coro 8-sekiz}, we also observed that for any nonempty topological space $X$, if $\Clop(X)$ is finite then it has at least a primitive idempotent. But there are infinite Boolean rings without having any primitive idempotents. We will provide such an example. First note that if $R$ is a Boolean ring then the map $e\mapsto D(e)$ is a bijective from the set of primitive idempotents of $R$ onto the set of isolated points of $\Spec(R)$ (in the same vein, the map $e\mapsto R(1-e)$ is also a bijection from the set of primitive idempotents of $R$ onto the set of finitely generated prime ideals of $R$). The space $2^{\mathbb{N}}$, the countably infinite product of the discrete space $\{0,1\}$, equipped with the product topology is a compact totally disconnected space. Then by Corollary \ref{Coro v-5}, there exists an (infinite) Boolean ring $R$ such that $\Spec(R)$ is homeomorphic to $2^{\mathbb{N}}$. But $R$ has no primitive idempotents, because it can be seen that $2^{\mathbb{N}}$ has no isolated points (see \cite[Remark 5.6]{A. Tarizadeh Racsam2}). In fact, for any natural number $n\geqslant0$ there exists an infinite Boolean ring having only $n$ primitive idempotents. Indeed, if $n\geqslant1$ then the set $X=2^{\mathbb{N}}\cup\{1,\ldots,n\}$ endowed with the topology consisting of all $U\cup V$ with $U$ an open subset of $2^{\mathbb{N}}$ and $V$ a subset of $\{1,\ldots,n\}$ is compact with the isolated points $\{1\},\ldots,\{n\}$. The space $X$ is also totally disconnected. Because for each $k\in\{1,\ldots,n\}$ the singleton $\{k\}$ is a clopen subset of $X$ and so it is a connected component of $X$. Let $C$ be a connected component of $X$ with $k\notin C$ for all $k\in\{1,\ldots,n\}$. The map $f:X\rightarrow2^{\mathbb{N}}$ defined by $f(x)=x$ for all $x\in2^{\mathbb{N}}$ and $f(k)=e_{k}$ is continuous, since $f^{-1}(U)=U\cup\{k: e_k\in U\}$. Then $C=f(C)$ is a connected subset of $2^{\mathbb{N}}$ and hence it is singleton. Then by Corollary \ref{Coro v-5}, there exists an infinite Boolean ring corresponding with $X$ that has exactly $n$ primitive idempotents. This observation also shows that the primitive idempotents, unlike the connected components, is not an invariant to check whether a given space has finitely many clopens or not.  
\end{example}

In the following result, the cardinal $\kappa=|\pi_{0}(X)|$ is not assumed to be finite. This is indeed the strengthening point of this theorem.
 
\begin{theorem}\label{Thm 6-six} Let $X$ be a quasi-compact space and  $\kappa=|\pi_{0}(X)|$. Then the following assertions are equivalent: \\
$\mathbf{(i)}$ $X$ has finitely many connected components. \\
$\mathbf{(ii)}$ We have an isomorphism of rings $\Clop(X)\simeq\prod\limits_{\kappa}
\mathbb{Z}/2$. \\
$\mathbf{(iii)}$ $\Clop(X)$ has the cardinality $2^\kappa$.\\
$\mathbf{(iv)}$ We have an isomorphism of rings $H_{0}(X)\simeq\prod\limits_{\kappa}\mathbb{Z}$. \\
$\mathbf{(v)}$ The connected components of $X$ are exactly the primitive idempotents of $\Clop(X)$. \\
$\mathbf{(vi)}$ Every connected component of $X$ is an open subset.
\end{theorem}

\begin{proof} The implications (i)$\Rightarrow$(ii) and (i)$\Rightarrow$(iv) follow from Theorem \ref{Theorem iii-uch}. \\
(ii)$\Rightarrow$(iii): There is nothing to prove. \\
(iii)$\Rightarrow$(i): It is well known that a Boolean ring $R$ is infinite if and only if $|R|\leqslant|\Spec(R)|$. Indeed, for the direct implication see \cite[Theorem 4.5]{A. Tarizadeh Racsam2} and the reverse implication is clear.   
Now if $R=\Clop(X)$ is infinite, then $|R|=2^\kappa\leqslant|\Spec(R)|$. But $X$ is quasi-compact and so by Lemma \ref{Lemma I}(iii),  $|\Spec(R)|\leqslant|\pi_{0}(X)|=\kappa$. It follows that $2^\kappa\leqslant\kappa$. But this contradicts Cantor's theorem (which asserts that $\kappa<2^\kappa$ for any cardinal $\kappa$). Thus $R=\Clop(X)$ is a finite ring and so $\Clop(X)\simeq(\mathbb{Z}/2)^{n}$ for some natural number $n\geqslant0$. Then the assertion follows from Theorem \ref{Theorem iii-uch}. \\
(iv)$\Rightarrow$(ii): By hypothesis, the Boolean ring of $H_{0}(X)$ is isomorphic to $\mathscr{B}(\mathbb{Z}^{\kappa})=(\mathbb{Z}/2)^{\kappa}$. The latter equality follows from the fact that for any direct product of rings $\prod\limits_{k}R_k$ we have $\mathscr{B}(\prod\limits_{k}R_k)=
\prod\limits_{k}\mathscr{B}(R_k)$. But $\Clop(X)$ is (canonically) isomorphic to the Boolean ring of $H_{0}(X)$. \\
(i)$\Rightarrow$(v): It is clear. \\
(v)$\Rightarrow$(i): Since $X$ is quasi-compact, there exist finitely many (distinct) primitive idempotents $A_1,\ldots,A_n\in\Clop(X)$ such that $X=\bigcup\limits_{k=1}^{n}A_k$. Hence, $\pi_{0}(X)=\{A_1,\ldots,A_n\}$. \\
(i)$\Leftrightarrow$(vi): It is clear. 
\end{proof} 

\begin{remark} In Theorem \ref{Thm 6-six},  
the assumption ``quasi-compactness of $X$" is necessary. Without this assumption, for instance, the implication  ``(iii)$\Rightarrow$(i)" will not be true. For example, consider the discrete topology over an infinite set $X$, then $\Clop(X)=\mathscr{P}(X)$ and the (discrete) space $X$ is totally disconnected, i.e., it is homeomorphic to $\pi_{0}(X)$. But $X$ has infinitely many connected components. Also note that in this theorem, the statement (iii) is weaker than statement (ii). In other words, there are examples of infinite Boolean rings $R$ with cardinality $|R|=2^\kappa$ for some infinite cardinal $\kappa$, but $R$ is not isomorphic to $(\mathbb{Z}/2)^\kappa$. Indeed, by \cite[Lemma 4.1]{A. Tarizadeh Racsam2}, there is an infinite Boolean ring $R$ with $|R|=|\Spec(R)|=2^\kappa$. But the ring $R$ is not isomorphic to $(\mathbb{Z}/2)^\kappa$. Because if $R\simeq(\mathbb{Z}/2)^\kappa\simeq\mathscr{P}(S)$ then their prime spectra will be homeomorphic: $\Spec(R)\simeq\Spec\big(\mathscr{P}(S)\big)$ where $S$ is an infinite set with $|S|=\kappa$. This, in particular, yields that they have the same cardinality equal to $2^\kappa$. But
it is well known that $\Spec\big(\mathscr{P}(S)\big)$ is the Stone-Čech compactification of the infinite discrete space $S$ and so its cardinality equals $2^{2^\kappa}$ which is a contradiction. In fact, for any infinite cardinal $\alpha$, this ring $R$ is not isomorphic to $(\mathbb{Z}/2)^\alpha$. Because if they are isomorphic, then we will have $2^\kappa=2^\alpha$. It follows that $\alpha=\kappa$, since otherwise we will have $\alpha<\kappa$ or  $\kappa<\alpha$. If $\alpha<\kappa$ then $\alpha<\kappa<2^\kappa=2^\alpha$ which contradicts the generalized continuum hypothesis. We get the same contradiction if $\kappa<\alpha$.         
\end{remark}

The following result considerably generalizes \cite[Theorem 5.7]{A. Tarizadeh Racsam2}. 

\begin{theorem}\label{Theorem ii-iki} For a nonzero ring $R$ the following assertions are equivalent: \\
$\mathbf{(i)}$ $\Spec(R)$ has finitely many connected components. \\ 
$\mathbf{(ii)}$ $R$ has finitely many idempotents. \\
$\mathbf{(iii)}$ There is an isomorphism of rings $R\simeq\prod\limits_{k=1}^{n}R/(1-e_{k})$ with $e_k$ is a primitive idempotent of $R$ and $n\geqslant1$ is a natural number. \\
$\mathbf{(iv)}$ There is an isomorphism of rings $R\simeq\prod\limits_{i\in\kappa}R/(1-e_{i})$ with $e_i$ is a primitive idempotent of $R$ and $\kappa$ is
the cardinal of connected components of $\Spec(R)$. \\
$\mathbf{(v)}$ The unit ideal of $R$ is generated by a set of its primitive idempotents. \\
$\mathbf{(vi)}$ Every regular ideal of $R$ is finitely generated. \\ 
$\mathbf{(vii)}$ Every max-regular ideal of $R$ is finitely generated. \\
$\mathbf{(viii)}$ The connected components of $\Spec(R)$ are exactly of the form $D(e)$ where $e$ is a primitive idempotent of $R$. 
\end{theorem}

\begin{proof}  
(i)$\Leftrightarrow$(ii): For any ring $R$ with $X=\Spec(R)$, then the map $\mathscr{B}(R)\rightarrow\Clop(X)$ given by $e\mapsto D(e)$ is an isomorphism of rings. 
Then apply Theorem \ref{Theorem iii-uch}. \\
(i)$\Rightarrow$(iii): By Theorem \ref{Thm 6-six}(v), there exist finitely many (distinct) primitive idempotents $e_1,\ldots,e_n\in R$ such that $\Spec(R)=\bigcup\limits_{k=1}^{n}D(e_k)$. Note that $n\geqslant1$, since $R$ is nonzero.  
Using Lemma \ref{Lemma iv}(ii), then we have $R(1-e_i)+R(1-e_k)=R$ for all $i\neq k$. It can be seen that $D(1)=\Spec(R)=\bigcup\limits_{k=1}^{n}D(e_{k})=
D(\sum\limits_{k=1}^{n}e_k)$. But $\sum\limits_{k=1}^{n}e_k$ is idempotent and so $\sum\limits_{k=1}^{n}e_k=1$. It follows that $\bigcap\limits_{k=1}^{n}R(1-e_k)=
\prod\limits_{k=1}^{n}R(1-e_k)=
R\big(\prod\limits_{k=1}^{n}(1-e_k)\big)=0$. Then by the Chinese remainder theorem, the ring map $R\rightarrow\prod\limits_{k=1}^{n}R/(1-e_{k})$ given by $r\mapsto\big(r+R(1-e_k)\big)_{k=1}^{n}$ is an isomorphism. \\
(iii)$\Rightarrow$(iv): By Lemmas \ref{Lemma ii} and  \ref{Lemma iv}, the nonzero ring $R/(1-e_k)$ has no nontrivial idempotents, and so $R$ has precisely $2^n$ idempotents. Then by Corollary \ref{Coro 8-sekiz}, $n$ is the number of connected components of $\Spec(R)$.  \\
(iv)$\Rightarrow$(i): By Lemmas \ref{Lemma ii} and  \ref{Lemma iv}, the nonzero ring $R/(1-e_k)$ has no nontrivial idempotents. Hence, the cardinal of idempotents of $R$ is precisely equal to $2^{\kappa}$. Then the assertion is deduced from Theorem \ref{Thm 6-six}(iii). \\
(i)$\Leftrightarrow$(v): It easily follows from Theorem \ref{Theorem iii-uch} and the fact that distinct primitive idempotents $e$ and $e'$ are orthogonal and hence $e\oplus e'=e+e'$. \\
(ii)$\Rightarrow$(vi)$\Rightarrow$(vii): There is nothing to prove. \\
(vii)$\Rightarrow$(viii): If $e$ is a primitive idempotent of $R$, then by Lemma \ref{Lemma iv}, $D(e)$ is a connected component of $\Spec(R)$. Conversely,
let $V(I)$ be a connected component of $\Spec(R)$ where $I$ is a max-regular ideal of $R$. By hypothesis, there exists an idempotent $e\in R$ such that $I=Re$. Then $V(I)=V(e)=D(1-e)$ and so $1-e$ is a primitive idempotent of $R$ (see Lemma \ref{Lemma iv}). \\
(viii)$\Rightarrow$(i): It is clear. 
\end{proof}

If a ring $R$ has finitely many idempotents, then by Theorem \ref{Theorem ii-iki}(iii),   $\{e_{1},\ldots,e_{n}\}$ is indeed the set of all (distinct) primitive idempotents of $R$. In particular, every Noetherian ring and more generally every ring $R$ with finitely many minimal primes has finitely many idempotents and there is a canonical isomorphism of rings $R\simeq R/(1-e_{1})\times\ldots\times R/(1-e_{n})$. 
Every ring $R$ can be canonically embedded in $\prod\limits_{M\in\Max(R)}R_{M}$. In particular, if $R$ has finitely many maximal ideals, then $R$ has finitely many idempotents and we have the same decomposition $R\simeq R/(1-e_{1})\times\ldots\times R/(1-e_{n})$.

\begin{remark} In Theorem \ref{Theorem ii-iki}(iv), the assumption ``$\kappa$ is the cardinal of connected components of $\Spec(R)$" is vital. Otherwise this theorem cannot be true. For example, take $R=\mathscr{P}(X)$ with $X$ an infinite set. Then $R\simeq\prod\limits_{x\in X}\mathbb{Z}/2\simeq\prod\limits_{x\in X}R/(1-e_x)$ where each $e_x=\{x\}$ is a primitive idempotent of $R$. Note that $|X|<2^{2^{|X|}}=$ the cardinal of connected components of $\Spec(R)$. This example also shows that if a ring has $2^\alpha$ idempotents with $\alpha=$ the cardinal of primitive idempotents, then the ring does not necessarily have finitely many idempotents. 
\end{remark} 

\textbf{Acknowledgment:} The author would like to express his deep thanks to Professor Brian Conrad and Professor Pierre Deligne, who generously shared their very valuable and excellent ideas with us during the writing of this article. 



\end{document}